\documentclass[11pt,oneside,english,reqno]{amsart}
\usepackage{amstext}
\usepackage{amsthm}
\usepackage{mathptmx}
\usepackage{helvet}
\usepackage{courier}
\usepackage{newtxmath}
\usepackage[T1]{fontenc}
\usepackage[utf8]{inputenc}
\usepackage{geometry}
\usepackage{color}
\usepackage{babel}
\usepackage{verbatim}
\usepackage{refstyle}
\usepackage{mathtools}
\usepackage[unicode=true,pdfusetitle,
 bookmarks=true,bookmarksnumbered=false,bookmarksopen=false,
 breaklinks=true,pdfborder={0 0 0},pdfborderstyle={},backref=false,colorlinks=true]
 {hyperref}
\hypersetup{
 pdfencoding=auto}

\makeatletter

\AtBeginDocument{\providecommand\eqref[1]{\ref{eq:#1}}}
\AtBeginDocument{\providecommand\thmref[1]{\ref{thm:#1}}}
\AtBeginDocument{\providecommand\secref[1]{\ref{sec:#1}}}
\AtBeginDocument{\providecommand\lemref[1]{\ref{lem:#1}}}
\AtBeginDocument{\providecommand\subsecref[1]{\ref{subsec:#1}}}
\AtBeginDocument{\providecommand\propref[1]{\ref{prop:#1}}}
\AtBeginDocument{\providecommand\corref[1]{\ref{cor:#1}}}
\RS@ifundefined{subsecref}
  {\newref{subsec}{name = \RSsectxt}}
  {}
\RS@ifundefined{thmref}
  {\def\RSthmtxt{theorem~}\newref{thm}{name = \RSthmtxt}}
  {}
\RS@ifundefined{lemref}
  {\def\RSlemtxt{lemma~}\newref{lem}{name = \RSlemtxt}}
  {}

\numberwithin{equation}{section}
\numberwithin{figure}{section}
\numberwithin{table}{section}
\theoremstyle{plain}
\newtheorem{thm}{\protect\theoremname}[section]
\theoremstyle{plain}
\newtheorem{prop}[thm]{\protect\propositionname}
\theoremstyle{plain}
\newtheorem{lem}[thm]{\protect\lemmaname}
\theoremstyle{plain}
\newtheorem{cor}[thm]{\protect\corollaryname}

\usepackage{refstyle}
\newref{rem}{name=Remark~,names=Remarks~}
\newref{lem}{name=Lemma~,names=Lemmas~}
\newref{def}{name=Definition~,names=Definitions~}
\newref{thm}{name=Theorem~,names=Theorems~}
\newref{prop}{name=Proposition~,names=Propositions~}
\newref{cond}{name=Condition~,names=Conditions~}
\newref{cor}{name=Corollary~,names=Corollaries~}
\newref{sec}{name=Section~,names=Sections~}
\newref{fig}{name=Figure~,names=Figures~}
\newref{subsec}{name=Section~,names=Sections~}
\newref{eq}{name=,refcmd=(\ref{#1})}
\hypersetup{final}
\providecommand{\TITLE}[1]{}
\TITLE{The continuum parabolic Anderson model with a half-Laplacian and periodic noise}
\providecommand{\AUTHORS}[1]{}
\AUTHORS{Alexander Dunlap\footnote{Department of Mathematics, Stanford University, Stanford, CA 94305
USA. \EMAIL{ajdunl2@stanford.edu}}}
\providecommand{\VOLUME}[1]{}
\VOLUME{XX}
\providecommand{\DOI}[1]{}
\DOI{XX}
\providecommand{\PAPERNUM}[1]{}
\PAPERNUM{XX}
\providecommand{\SHORTTITLE}[1]{}
\SHORTTITLE{Continuum PAM with a half-Laplacian}
\providecommand{\ABSTRACT}[1]{}
\global\long\def\eps{\varepsilon}
\ABSTRACT{We construct solutions of a renormalized continuum fractional parabolic Anderson model, formally given by $\partial_{t}u=-(-\Delta)^{\frac{1}{2}}u+\xi u$, where $\xi$ is a periodic spatial white noise. To be precise, we construct limits as $\eps\to0$ of solutions of $\partial_{t}u_{\eps}=-(-\Delta)^{\frac{1}{2}}u_{\eps}+(\xi_{\eps}-C_{\eps})u_{\eps}$, where $\xi_{\eps}$ is a mollification of $\xi$ at scale $\eps$ and $C_{\eps}$ is a logarithmically diverging renormalization constant. We use a simple renormalization scheme based on that of Hairer and Labbé, “A simple construction of the continuum parabolic Anderson model on $\mathbf{R}^{2}$.”}

\makeatother

\usepackage[style=trad-plain,isbn=false,doi=false,url=false]{biblatex}
\providecommand{\corollaryname}{Corollary}
\providecommand{\lemmaname}{Lemma}
\providecommand{\propositionname}{Proposition}
\providecommand{\theoremname}{Theorem}

\begin{document}
\global\long\def\supp{\operatorname{supp}}%

\global\long\def\Cov{\operatorname{Cov}}%

\global\long\def\Var{\operatorname{Var}}%

\global\long\def\pv{\operatorname{p.v.}}%

\global\long\def\e{\mathrm{e}}%

\global\long\def\R{\mathbf{R}}%

\global\long\def\supp{\operatorname{supp}}%

\global\long\def\dif{\mathrm{d}}%

\global\long\def\eps{\varepsilon}%

\ifdefined\journalversion
\excludecomment{abstract}
\newcommand{\title}[1]{}
\newcommand{\author}[1]{}
\newcommand{\date}[1]{}
\newcommand{\thanks}[1]{}
\newcommand{\email}[1]{}
\newcommand{\address}[1]{}
\fi
\title{The continuum parabolic Anderson model with a half-Laplacian and periodic
noise}
\author{Alexander Dunlap}
\date{\today}
\address{Department of Mathematics, Stanford University, Stanford, CA 94305
USA.}
\email{{\rmfamily\itshape \href{mailto:ajdunl2@stanford.edu}{\nolinkurl{ajdunl2@stanford.edu}}}}
\begin{abstract}
We construct solutions of a renormalized continuum fractional parabolic
Anderson model, formally given by $\partial_{t}u=-(-\Delta)^{\frac{1}{2}}u+\xi u$,
where $\xi$ is a periodic spatial white noise. To be precise, we
construct limits as $\eps\to0$ of solutions of $\partial_{t}u_{\eps}=-(-\Delta)^{\frac{1}{2}}u_{\eps}+(\xi_{\eps}-C_{\eps})u_{\eps}$,
where $\xi_{\eps}$ is a mollification of $\xi$ at scale $\eps$
and $C_{\eps}$ is a logarithmically diverging renormalization constant.
We use a simple renormalization scheme based on that of Hairer and
Labbé, ``A simple construction of the continuum parabolic Anderson
model on $\mathbf{R}^{2}$.''
\end{abstract}

\maketitle

\section{Introduction}

Let $\Lambda=-(-\Delta)^{\frac{1}{2}}$ be the half-Laplacian on $\mathbf{R}$.
It is given by the formula
\[
\Lambda f(x)=\frac{1}{\pi}\pv\int_{\mathbf{R}}\frac{f(y)-f(x)}{(y-x)^{2}}\,\dif y.
\]
Here and throughout the paper, $\pv\int_{\mathbf{R}}$ will denote
the principal value integral: if $g$ is a function with a singularity
at $x$, then
\[
\pv\int_{\mathbf{R}}g(y)\,\dif y=\lim_{\eps\downarrow0}\int_{\mathbf{R}\setminus[x-\eps,x+\eps]}g(y)\,\dif y.
\]
Also, let $\xi$ be a periodic Gaussian spatial white noise on $\mathbf{R}$
of period $L\in(0,\infty)$. The covariance kernel of $\xi$ is thus
given by $\mathbf{E}\xi(x)\xi(y)=\sum\limits _{k\in\mathbf{Z}}\delta(x-y+kL)$.
We are interested in the fractional parabolic Anderson model (PAM)
formally given by
\begin{align}
\partial_{t}u & =\Lambda u+\xi u; & u(0,\cdot) & =\underline{u}.\label{eq:fpam-formal}
\end{align}
We would expect solutions of \eqref{fpam-formal} to model scaling
limits of a PAM on the lattice with long-range jumps. This lattice
model, with a non-Gaussian noise, was previously studied in \cite{MZ12}.
We refer, for example, to \cite{K16} for more background on the PAM.

Straightforward heuristics indicate that \eqref{fpam-formal} cannot
be interpreted directly. Indeed, the white noise $\xi$ has (Hölder)
regularity ``$-\frac{1}{2}-$'' (i.e. any regularity strictly below
$-\frac{1}{2}$), and thus the solution to the linearized problem
(around $u\equiv1$) of \eqref{fpam-formal} has regularity $\frac{1}{2}-$,
since we gain one derivative by inverting the half-Laplacian. So we
can expect the regularity of $u$ to be at most $\frac{1}{2}-$. Thus
the product of $\xi$ and $u$ is undefined since the sum of their
regularities is (just barely) negative. This is why the power of $\frac{1}{2}$
on the Laplacian is interesting---it is the largest power such that
the product in \eqref{fpam-formal} is ill-defined.

Since abstract theory does not allow us to interpret the problem as
stated, we turn our attention to a regularized problem and try to
pass to a limit as the regularization is removed. We will see that
a renormalization is necessary to obtain a finite limit. Fix a mollifier
$\rho\in\mathcal{C}_{\mathrm{c}}^{\infty}$ (i.e. smooth with compact
support) so that $\int_{\mathbf{R}}\rho\equiv1$. For $\eps>0$, define
$\rho_{\eps}(x)=\eps^{-1}\rho(\eps^{-1}x)$ and $\xi_{\eps}=\rho_{\eps}*\xi$,
where $*$ denotes spatial convolution. Fix a constant $C_{\eps}\in\mathbf{R}$,
depending on $\eps$, and an initial condition $\underline{u}$. Then
we consider the problem
\begin{align}
\partial_{t}u_{\eps} & =\Lambda u_{\eps}+(\xi_{\eps}-C_{\eps})u_{\eps}; & u_{\eps}(0,\cdot) & =\underline{u}.\label{eq:uepsproblem}
\end{align}
This equation can be solved using standard techniques because $\xi_{\eps}\in\mathcal{C}^{\infty}$
for all $\eps>0$. Our goal will be to pass to the limit as $\eps\to0$.
To state our main theorem, we first define the Banach space in which
this convergence takes place. If $\mathcal{Y}$ is a Banach space,
we define for $\kappa\in\mathbf{R}$ and $T>0$ the Banach space $\mathcal{X}_{T}^{\kappa}(\mathcal{Y})$
to be the space of functions $f\in\mathcal{C}_{\mathrm{loc}}((0,T];\mathcal{Y})$
with finite norm
\begin{equation}
\|f\|_{\mathcal{X}_{T}^{\kappa}(\mathcal{Y})}\coloneqq\sup_{t\in(0,T]}t^{1-\kappa}\|f(t,\cdot)\|_{\mathcal{Y}}.\label{eq:XkappaTdef}
\end{equation}

\begin{thm}
\label{thm:maintheorem}There is a choice of deterministic constants
$C_{\eps}$, $\eps\in(0,1]$ (explicitly defined in \eqref{Cepsdef}
below), so that the following holds. For any $\kappa\in(0,1/4)$,
if $\underline{u}\in\mathcal{C}^{-\frac{1}{2}+2\kappa}$, then for
each $\eps\in[0,1]$ there is a random $u_{\eps}\in\mathcal{C}_{\mathrm{loc}}((0,\infty);\mathcal{C}^{\frac{1}{2}-\kappa})$
so that whenever $\eps>0$, $u_{\eps}$ is a mild solution to \eqref{uepsproblem},
and moreover for every $T>0$, $u_{\eps}\to u_{0}$ in probability
in $\mathcal{X}_{T}^{\kappa}(\mathcal{C}^{\frac{1}{2}-\kappa})$.
Finally, we have a constant $C<\infty$ so that, for all $\eps\in(0,1]$,
\begin{equation}
|C_{\eps}-(1/\pi)\log(1/\eps)|\le C.\label{eq:Cepsbound}
\end{equation}
\end{thm}

The model \eqref{fpam-formal} has similar local scaling properties
to the continuum PAM
\begin{equation}
\partial_{t}u=\Delta u+\xi u\label{eq:PAM}
\end{equation}
in \emph{two} spatial dimensions. That model also has a just-barely-ill-defined
product, and it also requires a logarithmic renormalization. Solutions
to \eqref{PAM} on a compact domain were constructed independently
in \cite{Hairer2014,GIP15} using the theories of regularity structures
and paracontrolled distributions, respectively. An elementary approach
that also works on the whole space was carried out by Hairer and Labbé
in \cite{HL15}, and some properties of solutions were derived in
\cite{GX18,GH19}. The more difficult case of \eqref{PAM} in three
spatial dimensions was tackled in \cite{HL16}. On the other hand,
singular stochastic PDEs involving fractional Laplacian terms have
previously been considered in \cite{BK17,CL19}.

Our approach to proving \thmref{maintheorem} closely follows the
strategy of \cite{HL15}, avoiding the use of regularity structures
or paracontrolled distributions. Similar strategies were used for
the random Schrödinger equation in \cite{DW18,DM19}. As in \cite{HL15},
we perform a change of variables in \eqref{uepsproblem} by writing
$u_{\eps}=\e^{S_{\eps}}v_{\eps}$, where $S_{\eps}$ is an approximate
solution to the linearized time-independent problem, and write a PDE
for $v_{\eps}$. (See \secref{chgvar}.) The coefficients of the PDE
for $v_{\eps}$ converge, in appropriate spaces, as $\eps\downarrow0$.
One of these converging ``coefficients'' is in fact a nonlocal operator.
Proving the convergence requires new estimates, which we carry out
in\textbf{ }\secref{coefficients} using some purely analytic bounds
that we prove in \secref{estimates}. Then the continuity of the PDE
for $v_{\eps}$ shows that $v_{\eps}\to v$, where $v$ solves the
limiting PDE. This is the main content of \secref{fixedpoint} and
is essentially the same as the argument of \cite{HL15}, as the estimates
having been obtained by this point are analogous. It is also easy
to see that $S_{\eps}$ converges to a limit $S$ as $\eps\downarrow0$.
Inverting the change of variables then shows that $u_{\eps}$ converges
to $\e^{-S}v$.

In this paper, we restrict ourselves to the case of periodic noise.
The periodicity is used so that the noise is bounded (as a distribution
in $\mathcal{C}^{-\frac{1}{2}-\kappa}$ for any $\kappa>0$) uniformly
in space. It is not clear whether or how solutions to \eqref{fpam-formal}
can be constructed with aperiodic white noise. In particular, the
weighted-space approach of \cite{HL15} does not immediately generalize
to our setting, because the Cauchy kernel decays only algebraically
in space, in contrast to the Gaussian decay of the heat kernel.

\subsection*{Acknowledgments}

We thank Lenya Ryzhik for suggesting the problem and much useful advice,
as well as Yu Gu, Leonid Mytnik, and Weijun Xu for interesting conversations.
We are also grateful to Leandro Chiarini for pointing out a subtlety
in the proof of \lemref{fractionalheatregularity}. We thank an anonymous
referee for a very careful reading of the manuscript and several important
comments. The author was partially supported by the NSF Graduate Research
Fellowship Program under Grant No.~DGE-1147470.

\section{Preliminaries and notation\label{sec:preliminaries}}

We will often work with constants, which we call $C$, and allow them
to change from line to line in a computation. This does \emph{not}
apply to the renormalization constant $C_{\eps}$, which will be fixed
in \eqref{Cepsdef} below.

\subsection{Hölder spaces}

We will work in $\alpha$-Hölder spaces, given as usual by the norm
\[
\|u\|_{\mathcal{C}^{\alpha}}=\|u\|_{L^{\infty}}+\sup_{|x-y|\le1}\frac{|u(x)-u(y)|}{|x-y|^{\alpha}}
\]
for all $\alpha\in(0,1)$. We will also use Hölder spaces with negative
Hölder exponent. Put
\begin{equation}
\eta_{x}^{\lambda}(y)=\lambda^{-1}\eta(\lambda^{-1}(y-x))\label{eq:etaable}
\end{equation}
for any function $\eta$. Then, for all $\alpha\in(-1,0)$, the $\alpha$-Hölder
norm of a distribution $u$ is
\[
\|u\|_{\mathcal{C}^{\alpha}}=\sup\{\lambda^{-\alpha}|u(\eta_{x}^{\lambda})|\ :\ x\in\mathbf{R},\eta\in\mathcal{C}^{1}([-1,1]),\|\eta\|_{\mathcal{C}^{0}}=1,\lambda\in(0,1]\}.
\]
Let $\mathcal{C}^{\alpha}$ be the Banach space of distributions such
that this norm is finite. We recall that $\mathcal{C}^{\alpha}$ is
equivalent to the Besov space $\mathcal{B}_{\infty,\infty}^{\alpha}$
(see \cite{BCD11}), and refer to \cite[Section 2]{HL15}, \cite[Section 3]{Hairer2014},
or \cite{FM17} for background on the use of negative Hölder spaces
for stochastic PDEs. We will use the following wavelet characterization
of negative Hölder spaces.
\begin{prop}[{\cite[Proposition 2.4]{HL15} or \cite[Definition 2.8 and Proposition 2.14]{FM17}}]
\label{prop:wavelets}There are compactly-supported functions $\psi,\phi\in\mathcal{C}_{\mathrm{c}}^{1}$
so that for any $\alpha\in(-1,0)$ we have a constant $C<\infty$
so that (using the notation \eqref{etaable})
\[
\|f\|_{\mathcal{C}^{\alpha}}\le C\sup_{x\in\mathbf{Z}}\left(\left|\int f\phi_{x}^{1}\right|+\sup_{n\in\mathbf{N}}2^{\alpha n}\left|\int_{\mathbf{R}}f\psi_{2^{-n}x}^{2^{-n}}\right|\right).
\]
\end{prop}

The following statement about multiplication of elements of Hölder
spaces is standard.
\begin{lem}[{\cite[Theorem 2.52]{BCD11}}]
\label{lem:multcont}If $\alpha<\beta$ and $\alpha+\beta>0$, then
multiplication of functions extends to a continuous bilinear map $\mathcal{C}^{\alpha}\times\mathcal{C}^{\beta}\to\mathcal{C}^{\alpha}$.
\end{lem}

\subsection{The fractional Laplacian}

In this section we establish some necessary background results on
the fractional Laplacian, especially on inverting the fractional Laplacian
$\Lambda$ and the fractional heat operator $\partial_{t}-\Lambda$.
We recall (see e.g. \cite{Kwa17}) the equivalent definition
\begin{equation}
\Lambda f(x)=\frac{1}{\pi}\int_{\mathbf{R}}\frac{f(y)-f(x)-(y-x)f'(x)\mathbf{1}\{|y-x|\le1\}}{(y-x)^{2}}\,\dif y.\label{eq:FL-equivalent-def}
\end{equation}
The regularization $(y-x)f'(x)\mathbf{1}\{|y-x|\le1\}$ obviates the
need for the principal value.

We let $\delta$ denote a Dirac delta distribution at $0$. We will
work with an approximate Green's function of the fractional Laplacian,
defined in the following lemma.
\begin{lem}
\label{lem:Gdef}There is a smooth even function $G:\mathbf{R}\setminus\{0\}\to\mathbf{R}$
so that $\supp G\subset[-1,1]$,
\begin{align}
G(x) & =(1/\pi)\log|x|\ \text{for all }x\in[-1/2,1/2],\label{eq:Gislog}
\end{align}
and if $F=\Lambda G-\delta$, then $F$ is smooth and there is a constant
$C$ so that, for all $x\in\mathbf{R}$,
\begin{equation}
|F(x)|,|F'(x)|\le C(1+|x|)^{-2}.\label{eq:Fbd}
\end{equation}
\end{lem}

\begin{proof}
Let $\widetilde{G}(x)=\frac{1}{\pi}\log|x|$ for all $x\in\mathbf{R}\setminus\{0\}$.
It is standard that $\Lambda\widetilde{G}=\delta$ in the sense of
distributions. Take $G$ to be any smooth even function $\mathbf{R}\setminus\{0\}\to\mathbf{R}$
such that $\supp G\subset[-1,1]$ and $G|_{[-\frac{1}{2},\frac{1}{2}]}=\widetilde{G}|_{[-\frac{1}{2},\frac{1}{2}]}$.
Define $F=\Lambda G-\delta=\Lambda(G-\widetilde{G})$. Since $G-\widetilde{G}$
is smooth, $F$ is smooth as well. The estimate \eqref{Fbd} is then
an easy consequence of the decay of the kernel in \eqref{FL-equivalent-def}.
\end{proof}

We also will need a Schauder-type estimate for $G$.
\begin{lem}
\label{lem:fractionalGreensregularity}If $\alpha\in(-1,0)$, there
is a $C<\infty$ so that if $f\in\mathcal{C}^{\alpha}$ then $G*f\in\mathcal{C}^{\alpha+1}$
and $\|G*f\|_{\mathcal{C}^{\alpha+1}}\le C\|f\|_{\mathcal{C}^{\alpha}}$.
\end{lem}

\begin{proof}
If $\eta$ is a smooth, positive function, supported on $[-\frac{1}{2},\frac{1}{2}]$,
identically $1$ in a neighborhood of $0$, then
\[
L(x)\coloneqq\int_{0}^{1}(1/y)\eta(x/y)\,\dif y=\int_{0}^{1/x}(1/y)\eta\left(1/y\right)\,\dif y=\pi G(x)+k(x)
\]
for some smooth, compactly-supported function $k$. Then it is sufficient
to prove that $\|L*f\|_{\mathcal{C}^{\alpha+1}}\le C\|f\|_{\mathcal{C}^{\alpha}}$.
To do this, we note first that $(L*f)(x)=\int_{0}^{1}\left(\frac{1}{y}\eta\left(\frac{\cdot}{y}\right)*f\right)(x)\,\dif y$.
Fix $x<x'$ and note that
\[
(L*f)(x)-(L*f)(x')=\int_{0}^{1}\int\frac{1}{y}q_{x,x',y}\left(\frac{z}{y}\right)f(z)\,\dif z\,\dif y,
\]
where $q_{x,x',y}(z)=\eta(x/y-z)-\eta(x'/y-z)$ and we use the common
abuse of notation in which the integrals in $z$ are in fact pairings
with the distribution $f$. If $y\ge2|x-x'|$, then $\|q_{x,x',y}\|_{\mathcal{C}^{0}}\le\|\eta\|_{\mathcal{C}^{1}}\frac{|x-x'|}{y}$
and $\supp q_{x,x',y}$ is contained in an interval of width $2$,
so $\left|\int\frac{1}{y}q_{x,x',y}(z/y)f(z)\,\dif z\right|\le\|\eta\|_{\mathcal{C}^{1}}\|f\|_{\mathcal{C}^{\alpha}}|x-x'|y^{\alpha-1}$.
On the other hand, if $0<y\le2|x-x'|$, then $q_{x,x',y}$ can be
written as the sum of two $\mathcal{C}^{1}$ functions with support
contained in $[-\frac{1}{2},\frac{1}{2}]$, each with $\mathcal{C}^{0}$
norm $\|\eta\|_{\mathcal{C}^{0}}$, so $\left|\int\frac{1}{y}q_{x,x',y}\left(\frac{z}{y}\right)f(z)\,\dif z\right|\le2\|\eta\|_{\mathcal{C}^{0}}\|f\|_{\mathcal{C}^{\alpha}}y^{\alpha}$.
Therefore, we have, for a constant $C$ depending on $\eta$ but not
on $f$, that
\begin{align*}
\left|(L*f)(x)-(L*f)(x')\right| & \le C\left(\int_{0}^{2|x-x'|}y^{\alpha}\,\dif y+|x-x'|\int_{2|x-x'|}^{1}y^{\alpha-1}\,\dif y\right)\|f\|_{\mathcal{C}^{\alpha}}\\
 & \le C|x-x'|^{\alpha+1}\|f\|_{\mathcal{C}^{\alpha}}.
\end{align*}
The necessary bound on $|(L*f)(x)|$ is easier, so we omit it.
\end{proof}
The inverse of the fractional heat operator $\partial_{t}-\Lambda$
is the Cauchy kernel $P_{t}(x)=\frac{t}{\pi(t^{2}+x^{2})}$. We will
need the following Schauder-type estimate for this kernel.
\begin{lem}
\label{lem:fractionalheatregularity}For any $T<\infty$ and $\alpha<\beta$,
there is a $C<\infty$ so that for any function $f\in\mathcal{C}^{\alpha}$
and any $t\in(0,T]$, we have $P_{t}*f\in\mathcal{C}^{\infty}$ and
$\|P_{t}*f\|_{\mathcal{C}^{\beta}}\le Ct^{-(\beta-\alpha)}\|f\|_{\mathcal{C}^{\alpha}}$.
\end{lem}

\begin{proof}
This follows from a scaling argument analogous to that used in \cite[Lemma 2.8]{HL15}.
For completeness, we present the argument for the case $-1<\alpha<0<\beta<1$,
which is what we use. As in \cite[Lemma 5.5]{Hairer2014},\footnote{We cannot quite apply \cite[Lemma 5.5]{Hairer2014} as stated, since
the Cauchy kernel cannot be extended by zero at negative times to
a smooth function on $\mathbf{R}^{2}\setminus\{0\}$. However, this
property is not necessary for our application.} fix a smooth function $\omega:[0,\infty)\to[0,\infty)$ so that $\supp\omega\subset[\frac{1}{2},2]$
and $\sum\limits _{n\in\mathbf{Z}}\omega(2^{n}\cdot)\equiv1$. Then
define, for $t\ge0$ and $x\in\mathbf{R}$, $P_{t}^{(n)}(x)\coloneqq\omega(2^{n}(t+|x|))P_{t}(x)$,
so we have $\sum\limits _{n\in\mathbf{Z}}P_{t}^{(n)}=P_{t}$ and (since
$P_{t}(x)=2^{n}P_{2^{n}t}(2^{n}x)$) $P_{t}^{(n)}(x)=2^{n}P_{2^{n}t}^{(0)}(2^{n}x)$.
Define $P_{t}^{-}=\sum\limits _{n<0}P_{t}^{(n)}$ and $P_{t}^{+}=\sum\limits _{n\ge0}P_{t}^{(n)}$.
We note that there is a constant $C<\infty$ so that for all $t\in(0,T]$,
we have the estimate $|P_{t}^{-}(x)|,|\partial_{x}P_{t}^{-}(x)|\le C(1+|x|)^{-2}$
for all $x\in\mathbf{R}$. This implies that $\|P_{t}^{-}*f\|_{\mathcal{C}^{\beta}}\le C\|f\|_{\mathcal{C}^{\alpha}}$.
On the other hand, we have $\|P_{t}^{(0)}\|_{\mathcal{C}^{2}}\le C$
for all $t\in(0,T]$. Therefore, we have that $\|2^{n}P_{t}^{(0)}(2^{n}\cdot)*f\|_{\mathcal{C}^{0}}\le C2^{-n\alpha}\|f\|_{\mathcal{C}^{\alpha}}$
and $\|\partial_{x}[2^{n}P_{t}^{(0)}(2^{n}\cdot)]*f\|_{\mathcal{C}^{0}}\le C2^{n(1-\alpha)}\|f\|_{\mathcal{C}^{\alpha}}$,
so
\[
\|2^{n}P_{t}^{(0)}(2^{n}\cdot)*f\|_{\mathcal{C}^{\beta}}\le C\|2^{n}P_{t}^{(0)}(2^{n}\cdot)*f\|_{\mathcal{C}^{0}}^{1-\beta}\|2^{n}P_{t}^{(0)}(2^{n}\cdot)*f\|_{\mathcal{C}^{1}}^{\beta}\le C2^{n(\beta-\alpha)}\|f\|_{\mathcal{C}^{\alpha}}.
\]
We now complete the proof by concluding that
\begin{align*}
\|P_{t}^{+}*f\|_{\mathcal{C}^{\beta}}\le C\sum_{n=0}^{\lceil-\log_{2}t\rceil+1}\|2^{n}P_{2^{n}t}^{(0)}(2^{n}\cdot)*f\|_{\mathcal{C}^{\beta}} & \le C\|f\|_{\mathcal{C}^{\alpha}}\sum_{n=0}^{\lceil-\log_{2}t\rceil+1}2^{n(\beta-\alpha)}\le\frac{C\|f\|_{\mathcal{C}^{\alpha}}}{t^{\beta-\alpha}}.\qedhere
\end{align*}
\end{proof}

\section{The change of variables\label{sec:chgvar}}

In this section we explain the key change of variables that we perform
on \eqref{uepsproblem}. This change of variables is an analogue for
the fractional Laplacian of the change of variables performed in \cite[p.3]{HL15}.
The advantage of the change of variables is that the coefficients
of the new equation converge as $\eps\downarrow0$, and so an equation
is obtained for the limit.
\begin{lem}
\label{lem:changeofvariables}For $\eps\ge0$, let $S_{\eps}=-G*\xi_{\eps}$,
where $G$ is defined as in \lemref{Gdef}. For $\eps>0$, if we put
$u_{\eps}=\e^{S_{\eps}}v_{\eps}$, then $v_{\eps}$ satisfies
\begin{align}
\partial_{t}v_{\eps} & =\Lambda v_{\eps}+v_{\eps}[-F*\xi_{\eps}+Z_{\eps}]+\Xi_{\eps}v_{\eps}, & v_{\eps}(0,\cdot) & =\e^{-S_{\eps}}\underline{u},\label{eq:vepsPDE}
\end{align}
where
\begin{gather}
Z_{\eps}(x)=\tilde{Z}_{\eps}(x)-C_{\eps},\qquad\tilde{Z}_{\eps}(x)=\frac{1}{\pi}\pv\int\frac{\e^{S_{\eps}(y)-S_{\eps}(x)}-(1+S_{\eps}(y)-S_{\eps}(x))}{(y-x)^{2}}\,\dif y,\label{eq:Zdef}\\
\Xi_{\eps}w(x)=\frac{1}{\pi}\pv\int\frac{(\e^{S_{\eps}(y)-S_{\eps}(x)}-1)(w(y)-w(x))}{(y-x)^{2}}\,\dif y.\label{eq:Xidef}
\end{gather}
\end{lem}

\begin{proof}
The initial condition is clear, so it remains to verify the PDE. We
note that
\[
\partial_{t}v_{\eps}=\e^{-S_{\eps}}\partial_{t}u_{\eps}=\e^{-S_{\eps}}\left[\Lambda u_{\eps}+(\xi_{\eps}-C_{\eps})u_{\eps}\right]=\e^{-S_{\eps}}\Lambda(\e^{S_{\eps}}v_{\eps})+(\xi_{\eps}-C_{\eps})v_{\eps}.
\]
It is straightforward to verify that
\[
\e^{-S_{\eps}}\Lambda(\e^{S_{\eps}}v_{\eps})=\Lambda v_{\eps}+v_{\eps}\Lambda S_{\eps}+\Xi_{\eps}v_{\eps}+\widetilde{Z}_{\eps}v_{\eps}.
\]
Also, $\Lambda S_{\eps}=-F*\xi_{\eps}-\xi_{\eps}$ by the definition
of $F$. Thus we have
\[
\partial_{t}v_{\eps}=\Lambda v_{\eps}+v_{\eps}\Lambda S_{\eps}+\Xi_{\eps}v_{\eps}+\widetilde{Z}_{\eps}v_{\eps}+(\xi_{\eps}-C_{\eps})v_{\eps}=\Lambda v_{\eps}+[-F*\xi_{\eps}+Z_{\eps}]v_{\eps}+\Xi_{\eps}v_{\eps}.\qedhere
\]
\end{proof}
The definitions of $S_{\eps}$ and $\Xi_{\eps}$ make sense for $\eps=0$
as well. We let $S=S_{0}$ and $\Xi=\Xi_{0}$.

\section{Analytic estimates\label{sec:estimates}}

In this section we derive some purely analytic estimates that will
help us control the quantities on the right side of \eqref{vepsPDE}.
Following \cite[Section 3]{HL15} or \cite[Section 10.3]{Hairer2014},
define the norm, for any $m\in\mathbf{N}$, $\zeta\in\mathbf{R}$,
and smooth function $K$ on $\mathbf{R}\setminus\{0\}$,
\begin{equation}
\|K\|_{\zeta;m}=\sup_{\substack{k\in\mathbf{Z}\\
0\le k\le m
}
}\sup_{x\in\mathbf{R}\setminus\{0\}}|x|^{k-\zeta}|K^{(k)}(x)|,\label{eq:singnorm}
\end{equation}
where $K^{(k)}$ denotes the $k$th derivative of $K$. We note in
particular that, with $G$ defined as in \lemref{Gdef}, we have
\begin{equation}
\|G\|_{-\kappa;m}<\infty\label{eq:Ganynegative}
\end{equation}
for all $\kappa>0$ and all $m\in\mathbf{N}$. We define the notation
\begin{equation}
\square K(\alpha;y,z)=K(\alpha)-K(\alpha-y)-K(\alpha-z)+K(\alpha-y-z).\label{eq:square}
\end{equation}
Quantities of this form arise in the expressions for moments of \eqref{Zdef}--\eqref{Xidef}.
\begin{lem}
For each $\theta\in(0,1)$, there is a constant $C<\infty$ so that
for any smooth function $K$ on $\mathbf{R}\setminus\{0\}$ and $\alpha,y,z\in\mathbf{R}$,
we have
\begin{equation}
\left|\square K(\alpha;y,z)\right|\le C\|K\|_{1-\theta;1}(|y|\wedge|z|)^{1-\theta}\label{eq:almost-lipschitz}
\end{equation}
If we further assume that $|y|,|z|<|\alpha|/4$, then 
\begin{equation}
|\square K(\alpha;y,z)|\le C|y||z|\|K\|_{1-\theta;2}|\alpha|^{-1-\theta}.\label{eq:closebound}
\end{equation}
\end{lem}

\begin{proof}
By the fundamental theorem of calculus and \eqref{singnorm}, we have
for $x<w$ that
\begin{align*}
|K(w)-K(x)|\le\int_{x}^{w}\|K\|_{1-\theta;1}|t|^{-\theta}\,\dif t & \le\|K\|_{1-\theta;1}\int_{-\frac{w-x}{2}}^{\frac{w-x}{2}}|t|^{-\theta}\,\dif t\le C\|K\|_{1-\theta;1}|w-x|^{1-\theta}.
\end{align*}
Thus by the triangle inequality we have
\[
|\square K(\alpha;y,z)|\le|K(\alpha)-K(\alpha-y)|+|K(\alpha-z)-K(\alpha-y-z)|\le C\|K\|_{1-\theta;1}|y|^{1-\theta},
\]
and similarly with $y$ and $z$ exchanged. This proves \eqref{almost-lipschitz}.
Now assume that $|y|,|z|<|\alpha|/4$. If $F(w,x)=K(\alpha-w-x)$,
then
\begin{align*}
|\square K(\alpha;y,z)| & =\left|\int_{0}^{y}\int_{0}^{z}K''(\alpha-w-x)\,\dif w\,\dif x\right|\\
 & \le\|K\|_{1-\theta;2}\left|\int_{0}^{y}\int_{0}^{z}(\alpha-w-x)^{-1-\theta}\,\dif w\,\dif x\right|\le C|y||z|\|K\|_{1-\theta;2}|\alpha|^{-1-\theta}.
\end{align*}
This proves \eqref{closebound}. 
\end{proof}
\begin{lem}
\label{lem:Hbound}For each $\theta\in(0,1)$ and all $M<\infty$,
there is a constant $C<\infty$ so that for any smooth functions $H_{1},H_{2}:\mathbf{R}\setminus\{0\}\to\mathbf{R}$
and all $\alpha\in\mathbf{R}$ we have
\begin{equation}
\int_{\mathbf{R}^{2}}\frac{\square H_{1}(\alpha;y,z)\cdot\square H_{2}(\alpha;y,z)}{y^{2}z^{2}}\,\dif y\,\dif z\le C\|H_{1}\|_{1-\theta;2}\|H_{2}\|_{1-\theta;2}|\alpha|^{-2\theta}.\label{eq:Ibound}
\end{equation}
\end{lem}

\begin{proof}
The left side of \eqref{Ibound} can be written as $I_{1}+I_{2}$,
where $I_{1}$ is the integral over the domain $\{|y|,|z|<|\alpha|/4\}$
and $I_{2}$ is the integral over the domain $\{|y|\vee|z|\ge\alpha/4\}$.
By \eqref{closebound},
\begin{equation}
I_{1}\le C\|H_{1}\|_{1-\theta;2}\|H_{2}\|_{1-\theta;2}|\alpha|^{-2\theta},\label{eq:I1bound}
\end{equation}
while by \eqref{almost-lipschitz},
\begin{equation}
\frac{I_{2}}{\|H_{1}\|_{1-\theta;1}\|H_{2}\|_{1-\theta;1}}\le C\int_{|\alpha|/4\le|y|\vee|z|}\frac{(|y|\wedge|z|)^{2(1-\theta)}}{y^{2}z^{2}}\,\dif y\,\dif z\le C|\alpha|^{-2\theta}.\label{eq:I2bound}
\end{equation}
Combining \eqref{I1bound} and \eqref{I2bound} yields \eqref{Ibound}.
\end{proof}

\section{Stability of the coefficients of the equation for \texorpdfstring{$v_\eps$}{v\_ε}\label{sec:coefficients}}

In this section we prove that the coefficients of the equation \eqref{vepsPDE}
are stable as we eliminate the spatial mollification of the noise.
We will consider the coefficients of \eqref{vepsPDE} in turn. The
stability of $F*\xi_{\eps}$ will come directly from the decay \eqref{Fbd}
of $F$ and $F'$. We will consider the term $Z_{\eps}$, which requires
renormalization, in \subsecref{Zeps}. We bound the size of the renormalization
constant $C_{\eps}$ in \subsecref{renormalization}. Then we show
the stability of the nonlocal operator $\Xi_{\eps}$ in \subsecref{Xieps}.
We will use the following two basic lemmas.
\begin{lem}
\label{lem:periodic-xi-regularity}For any $\kappa>0$ and $\eps\ge0$,
we have that $\xi_{\eps}\in\mathcal{C}^{-\frac{1}{2}-\kappa}$ almost
surely. Also, $\xi_{\eps}\to\xi$ as $\eps\to0$ in probability in
$\mathcal{C}^{-\frac{1}{2}-\kappa}$.
\end{lem}

\begin{proof}
This is a simple estimate using \propref{wavelets} as in \cite[Lemma 1.1]{HL15}.
\end{proof}
\begin{lem}
\label{lem:Sepsregularity}For any $\kappa>0$, we have for each $\eps\ge0$
that $S_{\eps}\in\mathcal{C}^{\frac{1}{2}-\kappa}$ almost surely.
(Recall that $S_{\eps}$ was defined in \lemref{changeofvariables}.)
Also, $S_{\eps}\to S$ as $\eps\to0$ in probability in $\mathcal{C}^{\frac{1}{2}-\kappa}$.
\end{lem}

\begin{proof}
This follows from \lemref{periodic-xi-regularity} by \lemref{fractionalGreensregularity},
similarly to \cite[Corollary 1.2]{HL15}.
\end{proof}

\subsection{Stability of \texorpdfstring{$Z_\eps$}{Z\_ε}\label{subsec:Zeps}}

For $\eps\ge0$, define $S_{\eps}(y,x)=S_{\eps}(y)-S_{\eps}(x)$.
(Norms of the form $\|S_{\eps}\|_{\bullet}$ will continue to refer
to the one-variable function $S_{\eps}$.) For $\eps>0$, we fix
\begin{equation}
C_{\eps}=\frac{1}{2\pi}\int\frac{\mathbf{E}S_{\eps}(y,0)^{2}}{y^{2}}\,\dif y.\label{eq:Cepsdef}
\end{equation}
We will prove in \propref{Ceps-value} below that the integrals on
the right side of \eqref{Cepsdef} are well-defined. For $\eta\in\mathcal{C}_{\mathrm{c}}^{1}$
and $\eps\ge0$, put
\begin{equation}
Z_{\eps}(\eta)=\frac{1}{\pi}\iint\frac{\eta(x)}{(y-x)^{2}}\left[\e^{S_{\eps}(y,x)}-\left(1+S_{\eps}(y,x)+\frac{1}{2}\mathbf{E}S_{\eps}(y,x)^{2}\right)\right]\,\dif y\,\dif x,\label{eq:Zepsdistndef}
\end{equation}
so $Z_{\eps}$ is a distribution. We will show that $Z_{\eps}\in\mathcal{C}^{-\kappa}$
for any $\kappa>0$ and $\eps\ge0$ in \propref{Zepsconv} below.
For $\eps>0$, \eqref{Zepsdistndef} agrees with \eqref{Zdef} with
the choice \eqref{Cepsdef} of $C_{\eps}$.

We split $Z_{\eps}$ into two parts. For $\eps\ge0$ and $\eta\in\mathcal{C}_{\mathrm{c}}^{1}$,
define
\begin{equation}
U_{\eps}(\eta)=\frac{1}{2\pi}\iint\frac{\eta(x)}{(y-x)^{2}}[S_{\eps}(y,x)^{2}-\mathbf{E}S_{\eps}(y,x)^{2}]\,\dif y\,\dif x,\label{eq:Uepsdef}
\end{equation}
and for $\eps\ge0$ and $x\in\mathbf{R}$ define
\begin{equation}
V_{\eps}(x)=\frac{1}{\pi}\int\frac{\e^{S_{\eps}(y,x)}-[1+S_{\eps}(y,x)+\frac{1}{2}S_{\eps}(y,x)^{2}]}{(y-x)^{2}}\,\dif y.\label{eq:Vepsdef}
\end{equation}
Now evidently $Z_{\eps}=U_{\eps}+V_{\eps}$. The main goal of this
section is to prove the following.
\begin{prop}
\label{prop:Zepsconv}Let $\kappa>0$. For every $\eps\ge0$, we have
$Z_{\eps}\in\mathcal{C}^{-\kappa}$ almost surely. Moreover, $Z_{\eps}\to Z_{0}$
in probability in $\mathcal{C}^{-\kappa}$.
\end{prop}

To prove \propref{Zepsconv}, we will show that $U_{\eps}$ and $V_{\eps}$
are both stable in $\mathcal{C}^{-\kappa}$ as $\eps\to0$. This is
the advantage of the change of variables carried out in \secref{chgvar},
since the coefficients of the original equation \eqref{uepsproblem}
do not converge as $\eps\to0$.

\subsubsection{Stability of \texorpdfstring{$V_\eps$}{V\_ε}}

First we consider $V_{\eps}$, which is actually stable in $L^{\infty}$
as $\eps\to0$.
\begin{lem}
\label{lem:Vebounded}For each $\eps\ge0$, we have $V_{\eps}\in L^{\infty}$
with probability $1$. Also, $V_{\eps}\to V_{0}$ in probability in
$L^{\infty}$.
\end{lem}

\begin{proof}
Let $f(x)=\e^{x}-(1+x+x^{2}/2)$. By Taylor's theorem, we have a $C<\infty$
so that
\begin{align}
|f(x)| & \le C(|x|^{3}\wedge1)\e^{|x|}; & |f'(x)| & \le C(|x|^{2}\wedge1)\e^{|x|}.\label{eq:fbd}
\end{align}
Thus, for all $\eps\ge0$ and all $x,y\in\mathbf{R}$,
\[
|\e^{S_{\eps}(y,x)}-(1+S_{\eps}(y,x)+\frac{1}{2}S_{\eps}(y,x)^{2})|\le C\exp\{2\|S_{\eps}\|_{\mathcal{C}^{\frac{1}{2}-\kappa}}\}(|y-x|^{\frac{3}{2}-3\kappa}\wedge1).
\]
But this means that
\[
|V_{\eps}(x)|\le C\exp\{2\|S_{\eps}\|_{\mathcal{C}^{\frac{1}{2}-\kappa}}\}\int\frac{|y-x|^{\frac{3}{2}-3\kappa}\wedge1}{(y-x)^{2}}\,\dif y,
\]
and the last integral is finite and independent of $x$ if $\kappa<\frac{1}{2}$.
This implies that $V_{\eps}\in L^{\infty}$.

Also, by the mean value theorem and \eqref{fbd}, we have for all
$x,y\in\mathbf{R}$ that
\begin{align*}
 & \left|\e^{S_{\eps}(y,x)}-(1+S_{\eps}(y,x)+\tfrac{1}{2}S_{\eps}(y,x)^{2})-[\e^{S(y,x)}-(1+S(y,x)+\tfrac{1}{2}S(y,x)^{2})]\right|\\
 & \qquad\le C|S_{\eps}(y,x)-S(y,x)|((|S_{\eps}(y,x)|\vee|S(y,x)|)^{2}\wedge1)\exp\{|S_{\eps}(y,x)|\vee|S(y,x)|\}\\
 & \qquad\le C\|S_{\eps}-S\|_{\mathcal{C}^{\frac{1}{2}-\kappa}}(|y-x|^{\frac{3}{2}-3\kappa}\wedge1)\exp\{2(\|S_{\eps}\|_{\mathcal{C}^{\frac{1}{2}-\kappa}}\vee\|S\|_{\mathcal{C}^{\frac{1}{2}-\kappa}})\}.
\end{align*}
Therefore, for all $x\in\mathbf{R}$, we have
\[
|V_{\eps}(x)-V(x)|\le C\|S_{\eps}-S\|_{\mathcal{C}^{\frac{1}{2}-\kappa}}\exp\{2(\|S_{\eps}\|_{\mathcal{C}^{\frac{1}{2}-\kappa}}\vee\|S\|_{\mathcal{C}^{\frac{1}{2}-\kappa}})\}\int\frac{|y-x|^{\frac{3}{2}-3\kappa}\wedge1}{(y-x)^{2}}\,\dif y.
\]
The integral is bounded independently of $x$, and by \lemref{Sepsregularity}
we have
\[
\|S_{\eps}-S\|_{\mathcal{C}^{\frac{1}{2}-\kappa}}\exp\{2(\|S_{\eps}\|_{\mathcal{C}^{\frac{1}{2}-\kappa}}\vee\|S\|_{\mathcal{C}^{\frac{1}{2}-\kappa}})\}\to0
\]
in probability as $\eps\to0$. This proves that $V_{\eps}\to V$ in
probability in $L^{\infty}$ as $\eps\to0$.
\end{proof}

\subsubsection{Stability of \texorpdfstring{$U_\eps$}{U\_ε}}

Now we show the stability of $U_{\eps}$. Since $U_{\eps}(\eta)$
is defined as an integral over squares of Gaussian random variables---elements
of the second Wiener chaos---we can use moment estimates to control
its regularity and establish its stability.
\begin{lem}
\label{lem:Uemomentbound}For each $\kappa>0$, we have a constant
$C$ so that for all $\eps\in[0,1]$ and $\eta\in\mathcal{C}^{1}([-1,1])$
with $\|\eta\|_{\mathcal{C}^{0}}\le1$ we have, defining $\eta_{x}^{\lambda}$
as in \eqref{etaable}, that $\mathbf{E}U_{\eps}(\eta_{x}^{\lambda})^{2}\le C\lambda^{-\kappa}$.
\end{lem}

\begin{proof}
Taking the second moment of \eqref{Uepsdef}, we have
\begin{equation}
\mathbf{E}U_{\eps}(\eta)^{2}=\frac{1}{\pi^{2}}\iiiint\frac{\eta(x)\eta(w)}{(x-y)^{2}(w-z)^{2}}\Cov\left(S_{\eps}(y,x)^{2},S_{\eps}(z,w)^{2}\right)\,\dif y\,\dif z\,\dif x\,\dif w.\label{eq:secondmomentofUeps}
\end{equation}
We can compute
\begin{align}
\mathbf{E}S_{\eps}(x,y)S_{\eps}(w,z) & =\sum_{k\in\mathbf{Z}}\int\left(G_{\eps}(x-r+kL)-G_{\eps}(y-r+kL)\right)\left(G_{\eps}(w-r)-G_{\eps}(z-r)\right)\,\dif r\nonumber \\
 & =\square H_{\eps}(x-w;x-y,z-w),\label{eq:exp-Seps-product}
\end{align}
where $H_{\eps}(q)=\sum\limits _{k\in\mathbf{Z}}G_{\eps}^{*2}(q+kL)$
and $\square$ is defined as in \eqref{square}. By the Isserlis theorem,
\[
\Cov\left(S_{\eps}(y,x)^{2},S_{\eps}(z,w)^{2}\right)=2(\mathbf{E}S_{\eps}(y,x)S_{\eps}(z,w))^{2}=2\left(\square H_{\eps}(x-w;x-y,z-w)\right)^{2}.
\]
Using \eqref{Ganynegative} and \cite[(10.12)]{Hairer2014} and recalling
that $G$ has compact support, we see that $\|H_{\eps}-H_{\eps}(0)\|_{1-\frac{\kappa}{2};m}\le C$,
where $C<\infty$ is a constant independent of $\eps$. But the $\square$
operator does not see constants, so applying \lemref{Hbound}, we
have
\begin{align*}
\mathbf{E}U_{\eps}(\eta)^{2} & \le\frac{2}{\pi^{2}}\iiiint\eta(x)\eta(w)\left(\frac{\square H_{\eps}(x-w;x-y,z-w)}{(x-y)(w-z)}\right)^{2}\,\dif y\,\dif z\,\dif x\,\dif w\\
 & \le C\iint\eta(x)\eta(w)|x-w|^{-\kappa}\,\dif x\,\dif w.
\end{align*}
Then the conclusion follows by rescaling.
\end{proof}
\begin{cor}
\label{cor:Uebounded}For any $\kappa>0$ and $\eps\ge0$, we have
$U_{\eps}\in\mathcal{C}^{-\kappa}$ almost surely.
\end{cor}

\begin{proof}
We note that $U_{\eps}$ is an element of the second Wiener chaos
by definition. By \lemref{Uemomentbound} and the equivalence of moments
of elements of finite Wiener chaoses (as stated for example in \cite[Lemma 10.5]{Hairer2014})
and \lemref{Uemomentbound}, for each $\kappa>0$ and $p\in[1,\infty)$
there is a constant $C=C(p,\kappa)<\infty$, depending only on $p$
and $\kappa$, so that
\begin{align*}
\mathbf{E}\sup_{\substack{x\in\mathbf{Z}\\
n\in\mathbf{N}
}
}2^{-\kappa np}|U_{\eps}(\psi_{2^{-n}x}^{2^{-n}})|^{p} & \le\sum_{n\in\mathbf{N}}\sum_{x=0}^{2^{n}L}2^{-\kappa np}\mathbf{E}|U_{\eps}(\psi_{2^{-n}x}^{2^{-n}})|^{p}\le CL\sum_{n\in\mathbf{N}}2^{n[1+\kappa-\kappa p/2]}.
\end{align*}
Choose $p>2(1/\kappa+1)$, so the last sum is finite. It is simpler
to show that $\mathbf{E}\sup\limits _{x\in\mathbf{Z}}|U_{\eps}(\phi_{x}^{1})|^{p}<\infty$.
By \propref{wavelets}, this means that $\mathbf{E}\|U_{\eps}\|_{\mathcal{C}^{-\kappa}}^{p}<\infty$,
so $\|U_{\eps}\|_{\mathcal{C}^{-\kappa}}<\infty$ almost surely.
\end{proof}
\begin{lem}
\label{lem:UepsminusUbound}For any $\kappa\in(0,1)$ and $R>0$,
we have a constant $C<\infty$ so that, for any $\eta\in\mathcal{C}^{1}([-R,R])$
with $\|\eta\|_{\mathcal{C}^{0}}\le1$ and any $x_{0}\in\mathbf{R}$,
we have
\begin{equation}
\mathbf{E}[U_{\eps}(\eta_{x_{0}}^{\lambda})-U(\eta_{x_{0}}^{\lambda})]^{2}\le C\eps^{\kappa/6}\lambda^{-\kappa}.\label{eq:Uepsconv}
\end{equation}
\end{lem}

\begin{proof}
For any $\eta\in\mathcal{C}_{\mathrm{c}}^{1}$, we have
\[
U_{\eps}(\eta)-U(\eta)=\frac{1}{2\pi}\iint\eta(x)\frac{S_{\eps}(y,x)^{2}-S(y,x)^{2}-\mathbf{E}\left(S_{\eps}(y,x)^{2}-S(y,x)^{2}\right)}{(y-x)^{2}}\,\dif y\,\dif x,
\]
so
\begin{equation}
\begin{aligned}\mathbf{E} & \left(U_{\eps}(\eta)-U(\eta)\right)^{2}=\frac{1}{4\pi^{2}}\iiiint\frac{\Cov\left(S_{\eps}(y,x)^{2}-S(y,x)^{2},S_{\eps}(z,w)^{2}-S(z,w)^{2}\right)}{(x-y)^{2}(w-z)^{2}}\\
 & \qquad\qquad\qquad\qquad\qquad\qquad\qquad\qquad\qquad\qquad\qquad\qquad\qquad\qquad\eta(x)\eta(w)\,\dif y\,\dif z\,\dif x\,\dif w.
\end{aligned}
\label{eq:vardif}
\end{equation}
By the Isserlis theorem we have that 
\begin{align*}
\Cov & \left(S_{\eps}(y,x)^{2}-S(y,x)^{2},S_{\eps}(z,w)^{2}-S(z,w)^{2}\right)\\
 & =\mathbf{E}\left[(S_{\eps}+S)(y,x)(S_{\eps}-S)(z,w)\right]\mathbf{E}\left[(S_{\eps}-S)(y,x)(S_{\eps}+S)(z,w)\right]\\
 & \qquad+\mathbf{E}\left[(S_{\eps}+S)(y,x)(S_{\eps}+S)(z,w)\right]\mathbf{E}\left[(S_{\eps}-S)(y,x)(S_{\eps}-S)(z,w)\right]\\
 & =\left(\left(\square\left[\mathcal{P}[(G_{\eps}+G)*(G_{\eps}-G)]\right]\right)^{2}+\square\left[\mathcal{P}[(G_{\eps}+G)^{*2}\right]\square\left[\mathcal{P}[(G_{\eps}-G)^{*2}]\right]\right)\\
 & \qquad\qquad\qquad\qquad\qquad\qquad\qquad\qquad\qquad\qquad\qquad\qquad\qquad(x-w;x-y,z-w),
\end{align*}
where $\mathcal{P}f(x)=\sum\limits _{k\in\mathbf{Z}}f(x+kL)$. Combining
this with \eqref{vardif} yields
\begin{align}
\mathbf{E} & [U(\eta)-U_{\eps}(\eta)]^{2}\le C\iiiint\big[\left(\square\left[\mathcal{P}[(G_{\eps}+G)*(G_{\eps}-G)\right](x-w;-y,z)\right)^{2}\nonumber \\
 & \qquad\qquad\qquad\qquad\qquad\qquad\qquad+\left(\square\left[\mathcal{P}[(G_{\eps}+G)^{*2}]\right]\cdot\square\left[\mathcal{P}[(G_{\eps}-G)^{*2}]\right]\right)(x-w;-y,z)\big]\nonumber \\
 & \qquad\qquad\qquad\qquad\qquad\qquad\qquad\qquad\qquad y^{-2}z^{-2}\eta(x)\eta(w)\,\dif y\,\dif z\,\dif x\,\dif w.\label{eq:penultimateUUeps}
\end{align}
By \cite[Lemma 10.17]{Hairer2014} and \eqref{Ganynegative}, for
any $m\in\mathbf{Z}_{\ge0}$ we have a constant $C$, independent
of $\eps$, so that $\|G_{\eps}\|_{-2\kappa;m}\le C$ and $\|G_{\eps}-G\|_{-2\kappa;m}\le C\eps^{\kappa}$.
By \cite[(10.12)]{Hairer2014}, we thus obtain
\begin{align*}
\|\mathcal{P}[(G_{\eps}-G)*(G_{\eps}+G)]-\mathcal{P}[(G_{\eps}-G)*(G_{\eps}+G)](0)\|_{1-3\kappa;m} & \le C\varepsilon^{\kappa};\\
\|\mathcal{P}[(G_{\eps}-G)^{*2}]-(G_{\eps}-G)^{*2}(0)\|_{1-3\kappa;m} & \le C\varepsilon^{2\kappa};\\
\|\mathcal{P}[(G_{\eps}+G)^{*2}]-(G_{\eps}+G)^{*2}(0)\|_{1-3\kappa;m} & \le C.
\end{align*}
The operator $\square$ does not see constants, so by these inequalities,
\eqref{penultimateUUeps}, and \lemref{Hbound},
\begin{align}
\mathbf{E}[U(\eta)-U_{\varepsilon}(\eta)]^{2} & \le C\varepsilon^{\kappa}\int\eta(x)\eta(w)|x-w|^{-6\kappa}\,\dif x\,\dif w.\label{eq:finalUUepsbound}
\end{align}
Then \eqref{Uepsconv} follows by substituting $\kappa\mapsto\kappa/6$
and rescaling. The assumption $\kappa<1$ is so that the integral
on the right side of \eqref{finalUUepsbound} can be bounded in terms
of $\|\eta\|_{\mathcal{C}^{0}}$ and $R$.
\end{proof}
\begin{cor}
\label{cor:UeconvtoU}For any $\kappa>0$, $U_{\varepsilon}\to U$
in probability in $\mathcal{C}^{-\kappa}$.
\end{cor}

\begin{proof}
Assume without loss of generality that $\kappa<1$. As in \corref{Uebounded},
$U_{\eps}-U$ is an element of the second Wiener chaos by definition.
By \lemref{UepsminusUbound} and the equivalence of moments of the
elements of finite Wiener chaoses, for each $\kappa>0$ and $p\in[1,\infty)$
there is a constant $C=C(p,\kappa)$, depending only on $p$ and $\kappa$,
so that 
\begin{align*}
\mathbf{E}\sup_{\substack{x\in\mathbf{Z}\\
n\in\mathbf{N}
}
}\frac{|(U_{\eps}-U)(\psi_{2^{n}x}^{2^{-n}})|^{p}}{2^{\kappa np}} & \le\sum_{n\in\mathbf{N}}\sum_{x=0}^{2^{n}L}\frac{\mathbf{E}|(U_{\eps}-U)(\psi_{2^{-n}x}^{2^{-n}})|^{p}}{2^{\kappa np}}\le CL\eps^{\frac{\kappa p}{6}}\sum_{n\in\mathbf{N}}2^{n[1+\kappa-\frac{\kappa p}{2}]}.
\end{align*}
Take $p>2(1/\kappa+1)$, so the last sum is finite. A simpler computation
shows that $\mathbf{E}\sup\limits _{x\in\mathbf{Z}}|(U_{\eps}-U)(\phi_{x}^{1})|^{p}\le C\eps^{\kappa p/6}$.
By \propref{wavelets}, $\mathbf{E}\|U_{\eps}-U\|_{\mathcal{C}^{-\kappa}}^{p}\le C\eps^{\kappa p/6}$
for some constant $C$ not depending on $\eps$, which means that
$U_{\eps}\to U$ in probability.
\end{proof}
The results of the last two subsections are now enough to prove \propref{Zepsconv}.
\begin{proof}[Proof of \propref{Zepsconv}.]
Since $Z_{\eps}=U_{\eps}+V_{\eps}$ for all $\eps\ge0$, the fact
that $Z_{\eps}\in\mathcal{C}^{-\kappa}$ almost surely is an immediate
consequence of \lemref{Vebounded} and \corref{Uebounded}, and the
convergence is an immediate consequence of \lemref[s]{Vebounded}
and \ref{lem:UepsminusUbound}.
\end{proof}

\subsection{The renormalization constant\label{subsec:renormalization}}

We now estimate the size of the renormalization constant $C_{\eps}$,
proving the bound \eqref{Cepsbound}.
\begin{prop}
\label{prop:Ceps-value}There is an absolute constant $C$ so that,
for all $\eps\in(0,1]$, $|C_{\eps}-(1/\pi)\log(1/\eps)|\le C$.
\end{prop}

\begin{proof}
By \eqref{exp-Seps-product}, we have (with $H_{\eps}$ as defined
there)
\begin{equation}
\mathbf{E}S_{\eps}(y,0)^{2}=\square H_{\eps}(0;-y,y)=2(H_{\eps}(0)-H_{\eps}(y)).\label{eq:ESepsy0}
\end{equation}
We note that $H_{\eps}$ is even so $H_{\eps}'(0)=0$. Thus, combining
\eqref{FL-equivalent-def}, \eqref{Cepsdef}, and \eqref{ESepsy0}
we obtain
\begin{equation}
C_{\eps}=-\Lambda H_{\eps}(0)=C-\Lambda(G^{*2}*\rho_{\eps}^{*2})(0)=C-\frac{1}{\pi}\log|\cdot|*\rho_{\eps}^{*2}(0)-F*G*\rho_{\eps}^{*2}(0).\label{eq:Cepsexpansion}
\end{equation}
Here, $C$ is independent of $\eps$, coming from the terms $k\ne0$
in the sum defining $H_{\eps}$. The third term in \eqref{Cepsexpansion}
is also bounded independently of $\eps$, while the second is
\[
-\frac{1}{\pi}\log|\cdot|*\rho_{\eps}^{*2}(0)=-\frac{1}{\pi}\int\log|x|\rho^{*2}(x/\eps)\,\dif x=-\frac{1}{\pi}\int\log|\eps x|\rho^{*2}(-x)\,\dif x=\frac{1}{\pi}\log\frac{1}{\eps}+C.\qedhere
\]
\end{proof}

\subsection{Stability of \texorpdfstring{$\Xi_\eps$}{Ξ\_ε}\label{subsec:Xieps}}

In this section, we show that $\Xi_{\eps}$ (defined in \eqref{Xidef})
is stable as $\eps\to0$.
\begin{lem}
\label{lem:Xicts}If $\kappa>0$ and $\eps\ge0$, the map $\Xi_{\eps}$
is almost surely a bounded linear map $\mathcal{C}^{\frac{1}{2}+\kappa}\to L^{\infty}$.
\end{lem}

\begin{proof}
We have 
\begin{equation}
\sup\limits _{x,y\in\mathbf{R}}\frac{|\e^{S_{\eps}(y,x)}-1|}{|y-x|^{\frac{1}{2}-\frac{\kappa}{2}}}\le1+\exp\{2\|S_{\eps}\|_{\mathcal{C}^{\frac{1}{2}-\frac{\kappa}{2}}}\},\label{eq:expbd}
\end{equation}
so by the triangle inequality, 
\[
|\Xi_{\eps}v(x)|\le C(1+\exp\{2\|S_{\eps}\|_{\mathcal{C}^{\frac{1}{2}-\frac{\kappa}{2}}}\})\|v\|_{\mathcal{C}^{\frac{1}{2}+\kappa}}\int\frac{|x-y|^{\frac{1}{2}-\frac{\kappa}{2}}(|x-y|^{\frac{1}{2}+\kappa}\wedge1)}{(y-x)^{2}}\,\dif y,
\]
and the integral is bounded independently of $x$.
\end{proof}
\begin{prop}
\label{prop:Xiconv}For any $\kappa>0$, we have $\Xi_{\varepsilon}\to\Xi$
in probability with respect to the norm topology of the space $\mathcal{B}(\mathcal{C}^{\frac{1}{2}+\kappa},L^{\infty})$
of bounded linear operators from $\mathcal{C}^{\frac{1}{2}+\kappa}$
to $L^{\infty}$.
\end{prop}

\begin{proof}
We have for all $x,y\in\mathbf{R}$ that
\[
|\e^{S_{\eps}(y,x)}-\e^{S(y,x)}|\le2\exp\{2\|S_{\eps}\|_{L^{\infty}}+2\|S\|_{L^{\infty}}\}\|S_{\eps}-S\|_{\mathcal{C}^{\frac{1}{2}-\frac{\kappa}{2}}}|y-x|^{\frac{1}{2}-\frac{\kappa}{2}},
\]
so by \lemref{Sepsregularity}, we have $\sup\limits _{x,y\in\mathbf{R}}\frac{|\e^{S_{\eps}(y,x)}-\e^{S(y,x)}|}{|y-x|^{\frac{1}{2}-\frac{\kappa}{2}}}\to0$
in probability as $\eps\to0$. Now we write, for any $v\in\mathcal{C}^{\frac{1}{2}+\kappa}$,
\begin{align*}
\left|(\Xi-\Xi_{\varepsilon})v(x)\right| & \le\frac{1}{\pi}\left|\int\frac{(e^{S_{\varepsilon}(y,x)}-e^{S(y,x)})\left(v(y)-v(x)\right)}{(y-x)^{2}}\,\dif y\right|\\
 & \le C\|v\|_{\mathcal{C}^{\frac{1}{2}+\kappa}}\left(\sup_{x,y\in\mathbf{R}}\frac{\left|e^{S_{\varepsilon}(y,x)}-e^{S(y,x)}\right|}{|y-x|^{\frac{1}{2}-\frac{\kappa}{2}}}\right)\int\frac{(|y-x|^{\frac{1}{2}+\kappa}\wedge1)|y-x|^{\frac{1}{2}-\frac{\kappa}{2}}}{(y-x)^{2}}\,\dif y.
\end{align*}
The right side is finite and independent of $v$, and as $\eps\to0$
converges to $0$ in probability.
\end{proof}

\section{The fixed-point argument\label{sec:fixedpoint}}

Fix $\kappa\in(0,1/4)$ and $T>0$, and define $\mathcal{X}_{T}^{\kappa}(\mathcal{Y})$
for any Banach space $\mathcal{Y}$ as in the introduction. We will
construct a solution to \eqref{vepsPDE} in the space $\mathcal{X}_{T}^{\kappa}(\mathcal{C}^{\frac{1}{2}+\kappa})$
using a fixed-point argument. For $g\in\mathcal{C}^{-\kappa}$, $\Psi\in\mathcal{B}(\mathcal{C}^{\frac{1}{2}+\kappa},\mathcal{C}^{-\kappa})$,
and $\underline{v}\in\mathcal{C}^{-\frac{1}{2}+2\kappa}$, define
the affine operator $\mathcal{M}_{g,\Psi,\underline{v}}$ on $\mathcal{X}_{T}^{\kappa}(\mathcal{C}^{\frac{1}{2}+\kappa})$
(see \corref{contraction} below) by
\[
\mathcal{M}_{g,\Psi,\underline{v}}v(t,x)=\mathcal{L}_{g,\Psi}v(t,x)+P_{t}*\underline{v}(x),\text{ with }\mathcal{L}_{g,\Psi}v(t,x)=\int_{0}^{t}P_{t-s}*\left(v(s,\cdot)\cdot g+\Psi v(s,\cdot)\right)\,\dif s.
\]
Recall that the fractional heat kernel $P_{t}$ was introduced before
\lemref{fractionalheatregularity}. Solutions $v_{\eps}$ of \eqref{vepsPDE}
are exactly fixed points of the map $\mathcal{M}_{Z_{\eps}-F*\xi_{\eps},\Xi_{\eps},\underline{v}_{\eps}}$.
We aim to show that there is a unique such fixed point. We start by
bounding the operator norm of $\mathcal{L}_{g,\Psi}$.
\begin{lem}
\label{lem:Lcts}We have a $C<\infty$ so that if $g\in\mathcal{C}^{-\kappa}$
and $\Psi\in\mathcal{B}(\mathcal{C}^{\frac{1}{2}+\kappa},\mathcal{C}^{-\kappa})$,
then $\mathcal{L}_{g,\Psi}\in\mathcal{B}(\mathcal{X}_{T}^{\kappa}(\mathcal{C}^{\frac{1}{2}+\kappa}))$
and $\|\mathcal{L}_{g,\Psi}\|_{\mathcal{B}(\mathcal{X}_{T}^{\kappa}(\mathcal{C}^{\frac{1}{2}+\kappa}))}\le C(\|g\|_{\mathcal{C}^{-\kappa}}+\|\Psi\|_{\mathcal{B}(\mathcal{C}^{\frac{1}{2}+\kappa},\mathcal{C}^{-\kappa})})T^{\frac{1}{2}-2\kappa}$.
\end{lem}

\begin{proof}
By \lemref[s]{fractionalheatregularity}~and~\ref{lem:multcont},
we have a constant $C<\infty$ so that, for $v\in\mathcal{X}_{T}^{\kappa}(\mathcal{C}^{\frac{1}{2}+\kappa})$,
\begin{align}
\|\mathcal{L}_{g,\Psi}v(t,\cdot)\|_{\mathcal{C}^{\frac{1}{2}+\kappa}} & \le C\int_{0}^{t}(t-s)^{-\frac{1}{2}-2\kappa}\left(\|v(s,\cdot)\cdot g\|_{\mathcal{C}^{-\kappa}}+\|\Psi v(s,\cdot)\|_{\mathcal{C}^{-\kappa}}\right)\,\dif s;\label{eq:useregularity}\\
\|v(s,\cdot)\cdot g\|_{\mathcal{C}^{-\kappa}} & \le Cs^{-1+\kappa}\|g\|_{\mathcal{C}^{-\kappa}}\|v\|_{\mathcal{X}_{T}^{\kappa}(\mathcal{C}^{\frac{1}{2}+\kappa})};\label{eq:vsgbound}\\
\|\Psi v(s,\cdot)\|_{\mathcal{C}^{-\kappa}} & \le s^{-1+\kappa}\|\Psi\|_{\mathcal{B}(\mathcal{C}^{\frac{1}{2}+\kappa},\mathcal{C}^{-\kappa})}\|v\|_{\mathcal{X}_{T}^{\kappa}(\mathcal{C}^{\frac{1}{2}+\kappa})}.\label{eq:Psibound}
\end{align}
Plugging \eqref{vsgbound} and \eqref{Psibound} into \eqref{useregularity}
and integrating (using that $\kappa<1/4$), we obtain
\begin{align*}
\|\mathcal{L}_{g,\Psi}v(t,\cdot)\|_{\mathcal{C}^{\frac{1}{2}+\kappa}} & \le C\|v\|_{\mathcal{X}_{T}^{\kappa}(\mathcal{C}^{\frac{1}{2}+\kappa})}(\|g\|_{\mathcal{C}^{-\kappa}}+\|\Psi\|_{\mathcal{B}(\mathcal{C}^{\frac{1}{2}+\kappa},\mathcal{C}^{-\kappa})})t^{-\frac{1}{2}-\kappa}.
\end{align*}
Then the conclusion follows from the definition \eqref{XkappaTdef}
of the $\mathcal{X}_{T}^{\kappa}(\mathcal{C}^{\frac{1}{2}+\kappa})$
norm.
\end{proof}
If $T$ is chosen sufficiently small, \lemref{Lcts} implies that
$\mathcal{M}_{g,\Psi,\underline{v}}$ is a contraction map:
\begin{cor}
\label{cor:contraction}There is a $C<\infty$ so that for any $g\in\mathcal{C}^{-\kappa}$,
$\Psi\in\mathcal{B}(\mathcal{C}^{\frac{1}{2}+\kappa},\mathcal{C}^{-\kappa})$,
and $\underline{v}\in\mathcal{C}^{-\frac{1}{2}+2\kappa}$, the map
$\mathcal{M}_{g,\Psi,\underline{v}}:\mathcal{X}_{T}^{\kappa}(\mathcal{C}^{\frac{1}{2}+\kappa})\to\mathcal{X}_{T}^{\kappa}(\mathcal{C}^{\frac{1}{2}+\kappa})$
is continuous, and if
\begin{equation}
T<[C(\|g\|_{\mathcal{C}^{-\kappa}}+\|\Psi\|_{\mathcal{B}(\mathcal{C}^{\frac{1}{2}+\kappa},\mathcal{C}^{-\kappa})})]^{-1/(1/2-2\kappa)}\eqqcolon T_{0}(\|g\|_{\mathcal{C}^{-\kappa}},\|\Psi\|_{\mathcal{B}(\mathcal{C}^{\frac{1}{2}+\kappa},\mathcal{C}^{-\kappa})}),\label{eq:Tsize}
\end{equation}
then $\mathcal{M}_{g,\Psi,\underline{v}}$ is a contraction map.
\end{cor}

\begin{proof}
If $v\in\mathcal{X}_{T}^{\kappa}(\mathcal{C}^{\frac{1}{2}+\kappa})$,
then we have $\mathcal{L}_{g,\Psi}v\in\mathcal{X}_{T}^{\kappa}(\mathcal{C}^{\frac{1}{2}+\kappa})$
by \lemref{Lcts}. By \lemref{fractionalheatregularity}, we have
a constant $C$ so that $\|P_{t}*\underline{v}\|_{\mathcal{C}^{\frac{1}{2}+\kappa}}\le Ct^{-1+\kappa}\|\underline{v}\|_{\mathcal{C}^{-\frac{1}{2}+2\kappa}}$.
This implies that $t\mapsto P_{t}*\underline{v}$ is an element of
$\mathcal{X}_{T}^{\kappa}(\mathcal{C}^{\frac{1}{2}+\kappa})$ as well.
Therefore, $\mathcal{M}_{g,\Psi,\underline{v}}v\in\mathcal{X}_{T}^{\kappa}(\mathcal{C}^{\frac{1}{2}+\kappa})$.
Since $\mathcal{M}_{g,\Psi,\underline{v}}v-\mathcal{M}_{g,\Psi,\underline{v}}\tilde{v}=\mathcal{L}_{g,\Psi}v-\mathcal{L}_{g,\Psi}\tilde{v}$,
the continuity and contraction come from \lemref{Lcts}.
\end{proof}
We now use the contraction mapping principle to construct fixed points
of $\mathcal{M}_{g,\Psi,\underline{v}}$.
\begin{lem}
For any $T<\infty$, $\mathcal{M}_{g,\Psi,\underline{v}}$ has a unique
fixed point $\mathcal{V}_{T}(g,\Psi,\underline{v})$ in $\mathcal{X}_{T}^{\kappa}(\mathcal{C}^{\frac{1}{2}+\kappa})$.
\end{lem}

\begin{proof}
This holds for $T<T_{0}(g,\Psi)$ by \corref{contraction}. As $T_{0}(g,\Psi)$
does not depend on $\underline{v}$, the construction can be extended
to all $T$ as in the proof of \cite[Proposition 4.1]{HL15}.
\end{proof}
Now, as in \cite[Proposition 4.2]{HL15}, we show that the solution
map $\mathcal{V}_{T}$ is continuous using a mild solution argument.
\begin{prop}
\label{prop:fixed-point-continuous}For $T<\infty$, the map $\mathcal{V}_{T}:\mathcal{C}^{-\kappa}\times\mathcal{B}(\mathcal{C}^{\frac{1}{2}+\kappa},\mathcal{C}^{-\kappa})\times\mathcal{C}^{-\frac{1}{2}+2\kappa}\to\mathcal{X}_{T}^{\kappa}(\mathcal{C}^{\frac{1}{2}+\kappa})$
is continuous.
\end{prop}

\begin{proof}
Let $M>0$ be arbitrary. We will show that $\mathcal{V}_{T}$ is continuous
on $A_{M}=\{(g,\Psi,\underline{v})\;:\;\|g\|_{\mathcal{C}^{-\kappa}},\|\Psi\|_{\mathcal{B}(\mathcal{C}^{\frac{1}{2}+\kappa},\mathcal{C}^{-\kappa})}\le M\}$.
First suppose that $T<T_{0}(M,M)$. For $(g,\Psi,\underline{v}),(\tilde{g},\tilde{\Psi},\underline{\tilde{v}})\in A_{M}$,
put $v=\mathcal{V}_{T}(g,\Psi,\underline{v})$ and $\tilde{v}=\mathcal{V}_{T}(\tilde{g},\tilde{\Psi},\underline{\tilde{v}})$,
so
\begin{align*}
(v-\tilde{v})(t,\cdot) & =\mathcal{L}_{g,\Psi}(v-\tilde{v})(t,\cdot)+\int_{0}^{t}P_{t-s}*[\tilde{v}(s,\cdot)(g-\tilde{g})+(\Psi-\tilde{\Psi})\tilde{v}(s,\cdot)]\,\dif s+P_{t}*(\underline{v}-\underline{\tilde{v}}).
\end{align*}
Thus for all $t\in(0,T]$ we have
\begin{multline*}
\|(v-\tilde{v})(t,\cdot)\|_{\mathcal{C}^{\frac{1}{2}+\kappa}}\le\|\mathcal{L}_{g,\Psi}(v-\tilde{v})(t,\cdot)\|_{\mathcal{C}^{\frac{1}{2}+\kappa}}+Ct^{-1+\kappa}\|\underline{v}-\underline{\tilde{v}}\|_{\mathcal{C}^{-\frac{1}{2}+2\kappa}}\\
+C(\|g-\tilde{g}\|_{\mathcal{C}^{-\kappa}}+\|\Psi-\tilde{\Psi}\|_{\mathcal{B}(\mathcal{C}^{\frac{1}{2}+\kappa},\mathcal{C}^{-\kappa})})\int_{0}^{t}(t-s)^{-\frac{1}{2}-2\kappa}\|\tilde{v}(s,\cdot)\|_{\mathcal{C}^{\frac{1}{2}+\kappa}}\,\dif s.
\end{multline*}
Therefore,
\begin{align*}
\|v-\tilde{v}\|_{\mathcal{X}_{T}^{\kappa}(\mathcal{C}^{\frac{1}{2}+\kappa})} & \le\|\mathcal{L}_{g,\Psi}\|_{\mathcal{B}(\mathcal{X}_{T}^{\kappa}(\mathcal{C}^{\frac{1}{2}+\kappa}))}\|v-\tilde{v}\|_{\mathcal{X}_{T}^{\kappa}(\mathcal{C}^{\frac{1}{2}+\kappa})}+C\|\underline{v}-\underline{\tilde{v}}\|_{\mathcal{C}^{-\frac{1}{2}+2\kappa}}\\
 & \quad+C(\|g-\tilde{g}\|_{\mathcal{C}^{-\kappa}}+\|\Psi-\tilde{\Psi}\|_{\mathcal{B}(\mathcal{C}^{\frac{1}{2}+\kappa},\mathcal{C}^{-\kappa})})\|\tilde{v}\|_{\mathcal{X}_{T}^{\kappa}(\mathcal{C}^{\frac{1}{2}+\kappa})}\int_{0}^{t}\frac{s^{-1+\kappa}}{(t-s)^{\frac{1}{2}+2\kappa}}\,\dif s\\
 & \le\|\mathcal{L}_{g,\Psi}\|_{\mathcal{B}(\mathcal{X}_{T}^{\kappa}(\mathcal{C}^{\frac{1}{2}+\kappa}))}\|v-\tilde{v}\|_{\mathcal{X}_{T}^{\kappa}(\mathcal{C}^{\frac{1}{2}+\kappa})}+C\|\underline{v}-\underline{\tilde{v}}\|_{\mathcal{C}^{-\frac{1}{2}+2\kappa}}\\
 & \quad+CM(\|g-\tilde{g}\|_{\mathcal{C}^{-\kappa}}+\|\Psi-\tilde{\Psi}\|_{\mathcal{B}(\mathcal{C}^{\frac{1}{2}+\kappa},\mathcal{C}^{-\kappa})})T^{\frac{1}{2}-\kappa}.
\end{align*}
Since $T<T_{0}(M,M)$, we have $\|\mathcal{L}_{g,\Psi}\|_{\mathcal{B}(\mathcal{X}_{T}^{\kappa}(\mathcal{C}^{\frac{1}{2}+\kappa}))}<1$,
so 
\[
\|v-\tilde{v}\|_{\mathcal{X}_{T}^{\kappa}(\mathcal{C}^{\frac{1}{2}+\kappa})}\le C\cdot\frac{\|\underline{v}-\underline{\tilde{v}}\|_{\mathcal{C}^{-\frac{1}{2}+2\kappa}}+M(\|g-\tilde{g}\|_{\mathcal{C}^{-\kappa}}+\|\Psi-\tilde{\Psi}\|_{\mathcal{B}(\mathcal{C}^{\frac{1}{2}+\kappa},\mathcal{C}^{-\kappa})})T^{\frac{1}{2}-\kappa}}{1-\|\mathcal{L}_{g,\Psi}\|_{\mathcal{B}(\mathcal{X}_{T}^{\kappa}(\mathcal{C}^{\frac{1}{2}+\kappa}))}}.
\]
This shows that $\mathcal{V}_{T}$ is continuous on $A_{M}$ when
$T<T_{0}(M,M)$.

Now suppose (as an inductive hypothesis) that $\mathcal{V}_{T}$ is
continuous on $A_{M}$ and let $T'\in(T,T+T_{0}(M,M))$. For $(g,\Psi,\underline{v}),(\tilde{g},\tilde{\Psi},\underline{\tilde{v}})\in A_{M}$,
put $v=\mathcal{V}_{T}(g,\Psi,\underline{v})$ and $\tilde{v}=\mathcal{V}_{T}(\tilde{g},\tilde{\Psi},\underline{\tilde{v}})$.
For $t\in[T,T']$, let $w(t,x)=v(t-T,x)$ and let $\tilde{w}(t,x)=\tilde{v}(t-T,x)$.
We have $w=\mathcal{V}_{T'-T}(g,\Psi,v(T,\cdot))$ and $\tilde{w}=\mathcal{V}_{T'-T}(g,\Psi,\tilde{v}(T,\cdot))$.
Since $\mathcal{V}_{T'-T}$ is continuous, so is $\mathcal{V}_{T'}$.

By induction, this implies that $\mathcal{V}_{T}$ is continuous on
$A_{M}$ for any $T$. Since this is true for any $M$, we see that,
for any $T$, $\mathcal{V}_{T}$ is in fact continuous on $\mathcal{C}^{-\kappa}\times\mathcal{B}(\mathcal{C}^{\frac{1}{2}+\kappa},\mathcal{C}^{-\kappa})\times\mathcal{C}^{-\frac{1}{2}+2\kappa}$.
\end{proof}
The continuity of the solution map then allows us to apply the stability
results proved in \secref{coefficients} to show that the solutions
converge.
\begin{cor}
\label{cor:vepsconv}If $\underline{u}\in\mathcal{C}^{-\frac{1}{2}+2\kappa}$,
then for all $\eps\in[0,1)$, for any $T>0$ there is a unique solution
$v_{\eps}\in\mathcal{X}_{T}^{\kappa}(\mathcal{C}^{\frac{1}{2}+\kappa})$
to \eqref{vepsPDE}, and $v_{\eps}$ converges to $v_{0}$ in probability
in $\mathcal{X}_{T}^{\kappa}(\mathcal{C}^{\frac{1}{2}+\kappa})$.
\end{cor}

\begin{proof}
As noted above, $v_{\eps}=\mathcal{V}_{T}(Z_{\eps}-F*\xi_{\eps},\Xi_{\eps},\e^{-S_{\eps}}\underline{u})$
uniquely solves \eqref{vepsPDE}. Using \eqref{Fbd}, the periodicity
of $\xi$ and $\xi_{\eps}$, and \lemref{periodic-xi-regularity},
we see that $F*\xi_{\eps}\to F*\xi$ as $\eps\to0$ in probability
in $L^{\infty}$. \propref{Zepsconv} says that $Z_{\eps}\to Z_{0}$
in probability in $\mathcal{C}^{-\kappa}$, and \propref{Xiconv}
says that $\Xi_{\eps}\to\Xi_{0}$ in probability in $\mathcal{B}(\mathcal{C}^{\frac{1}{2}+\kappa},\mathcal{C}^{-\kappa})$.
Also, \lemref[s]{Sepsregularity}~and~\ref{lem:multcont}, along
with the assumption $\underline{u}\in\mathcal{C}^{-\frac{1}{2}+2\kappa}$,
imply that $\e^{-S_{\eps}}\underline{u}\to\e^{-S_{0}}\underline{u}$
in probability in $\mathcal{C}^{-\frac{1}{2}+2\kappa}$. Thus \propref{fixed-point-continuous}
implies that $v_{\eps}\to v_{0}$ in probability in $\mathcal{X}_{T}^{\kappa}(\mathcal{C}^{\frac{1}{2}+\kappa})$.
\end{proof}
To prove \thmref{maintheorem}, it simply remains to undo the change
of variables.
\begin{proof}[Proof of \thmref{maintheorem}.]
For $\eps>0$, we have $u_{\eps}=\e^{S_{\eps}}v_{\eps}$ by \lemref{changeofvariables}.
\lemref{Sepsregularity}, \corref{vepsconv}, and \lemref{multcont}
imply that as $\eps\downarrow0$, $u_{\eps}$ converges in probability
to $\e^{S_{0}}v_{0}$ in $\mathcal{X}_{T}^{\kappa}(\mathcal{C}^{\frac{1}{2}-\kappa})$.
The estimate \eqref{Cepsbound} for $C_{\eps}$ was proved as \propref{Ceps-value}.
\end{proof}
\emergencystretch=1em

\printbibliography

\end{document}